\documentclass[12pt]{amsart}

\usepackage{amssymb}

\usepackage{enumitem}

\usepackage{graphicx}

\makeatletter
\@namedef{subjclassname@2020}{%
  \textup{2020} Mathematics Subject Classification}
\makeatother

\usepackage[T1]{fontenc}

\newtheorem{theorem}{Theorem}[section]

\newtheorem{lemma}[theorem]{Lemma}

\theoremstyle{definition}

\numberwithin{equation}{section}

\usepackage{bm}

\usepackage{color}

\DeclareMathOperator{\supp}{supp}

\begin{document}

\title[On Bass' conjecture of the small Davenport constant]{On Bass' conjecture of the small Davenport constant}

\author[G. Wang]{Guoqing Wang}
\address{School of Mathematical Sciences\\ Tiangong University\\
Tianjin 300387, P.R.
China}
\email{gqwang1979@aliyun.com}

\author[Y. Zhao]{Yang Zhao}
\address{School of Mathematical Sciences\\ Tiangong University\\
Tianjin 300387, P.R.
China}
\email{Zhyang202412@163.com}

\date{}

\begin{abstract}
Let $G$ be a finite group. The small Davenport constant $\mathsf d(G)$ of $G$ is the maximal integer $\ell$ such that there is a sequence of length $\ell$ over $G$ which has no nonempty product-one subsequence. In 2007, Bass conjectured that $\mathsf d(G_{m,n})=m+n-2$, where $G_{m,n}=\langle x, y| x^m=y^n=1, x^{-1}yx=y^s\rangle$, and $s$ has order $m$ modulo $n$. In this paper, we confirm the conjecture for any group $G_{m,n}$ with additional conditions that $s$ has order $m$ modulo $q$, for every prime divisor $q$ of $n$. Moreover, we solve the associated inverse problem characterizing the structure of any product-one free sequence with extremal length $\mathsf d(G_{m,n})$. Our results  generalize some obtained theorems on this problem.
\end{abstract}

\subjclass[2020]{Primary 11B75; Secondary 11P70}

\keywords{Bass' conjecture, small Davenport constant, Product-one sequence, Zero-sum, Metacyclic group}

\maketitle

\section{Introduction and main results}
Let $G$ be a finite group written multiplicatively with the operation $*$. The identity element $1_G$ of $G$ is denoted as $1$ for short, and the usage will not ambiguous in context.  Let $S=g_1\bm\cdot \ldots\bm\cdot g_{\ell}$ be a sequence over $G$ with length $\ell$. We use
$$\pi(S)=\{g_{\tau (1)}*\cdots *g_{\tau (\ell)}: \tau \mbox { is a permutation of } [1,\ell]\}\subseteq G$$
to denote the set of products of $S$. We say that $S$ is a {\sl product-one sequence} if $1\in\pi(S)$, and that $S$ is a {\sl minimal product-one sequence} if it is a product-one sequence which can not be partitioned into two nontrivial, product-one subsequences. The small Davenport constant of $G$, denoted $\mathsf d(G)$, is the maximal integer $\ell$ such that there is a sequence over $G$ of length $\ell$ which has no nonempty product-one subsequence. The large Davenport constant of $G$, denoted $\mathsf D(G)$, is the maximal length of all minimal product-one sequences over $G$. A simple argument \cite[Lemma 2.4]{GG2013} shows that $\mathsf d(G)+1\leq \mathsf D(G)\leq |G|$ with equality in the first bound when $G$ is abelian, and equality in the second bound when $G$ is cyclic. In 2022, Y. Qu, Y. Li and D. Teeuwsen \cite{QLT} proved that $\mathsf d(G)\leq |G|/p+p-2$ for any finite non-cyclic group $G$, where $p$ is the smallest prime divisor of $|G|$, which solved a conjecture of Gao, Li and Peng \cite{GLP2014}. Moreover, they \cite{QLT2} also proved that the equality holds if and only if $G$ is a non-cyclic group with a cyclic subgroup of index $p$.

Besides the study of the bounds for small Davenport constant in general finite groups, finding its precise values in some important specific types of non-abelian finite groups also attracted some researchers.
Recent research on the precise value of the small Davenport constant and the associated inverse problem was mainly focused on metacyclic groups, denoted $G_{m,n}$. In 2007, Bass \cite{Bass2007} conjectured that $\mathsf d(G_{m,n})=m+n-2$, where $G_{m,n}=\langle x, y| x^m=y^n=1, x^{-1}yx=y^s\rangle$, and $s$ has order $m$ modulo $n$. He proved the conjecture is true when $n$ is a prime. In 2019, F. Brochero Mart\'{\i}nez and S. Ribas \cite{BR2019} determined the form of all product-one free sequences $S$ over $G_{m,q}$ with length $|S|=m+q-2$, where $q$ is a prime. In 2023, Y. Qu and Y. Li \cite{QL} proved that $\mathsf d(G_{p,n})=p+n-2$ with condition $\gcd((s-1)p,n)=1$ for prime $p$, and furthermore, they \cite{QL2022}  determined the structure of all product-one free sequences $S$ over $G_{p,n}$ of length $|S|=p+n-2$. More related results can be found in \cite{ABR2023-2,BR2022,BLMR,HZ2019,YZF}.

Denote $G_{m,n}^\star=\langle x, y| x^m=y^n=1, x^{-1}yx=y^s\rangle$, where $s$ has order $m$ modulo $q$ for every prime divisor $q$ of $n$. In this paper, we prove that Bass' conjecture of the small Davenport constant is true for all the groups $G_{m,n}^\star$ and solve the associated inverse problem.

Our main results are as follows.

\begin{theorem}\label{d(Gmnstar)}  $\mathsf d(G_{m,n}^\star)=m+n-2$.
\end{theorem}

\begin{theorem}\label{inversed(G)} Let $S$ be a sequence over $G_{m,n}^\star$ with $|S|=\mathsf d(G_{m,n}^\star)=m+n-2$. Then $S$ is product-one free if and only if $S$ is of one of the following forms:
\begin{enumerate}
\item[(a)] If $G_{m,n}^\star\not\cong G_{2,3}$ then $$S=(x^ay^{b_1})\bm\cdot\ldots\bm\cdot (x^ay^{b_{m-1}})\bm\cdot (y^c)^{[n-1]},$$
where $a,c,b_1,\ldots,b_{m-1}\in \mathbb{N}$ with $\mbox{gcd}(a,m)=\mbox{gcd}(c,n)=1$;

\item[(b)] If $G_{m,n}^\star\cong G_{2,3}$ then $S=x\bm\cdot xy\bm\cdot xy^2$, or $S= xy^{b}\bm\cdot (y^{c})^{[2]}$ where $b\in [0,2]$ and $c\in[1,2]$.
\end{enumerate}
\end{theorem}

We remark that the two classes of groups mentioned above ($G_{m,q}$ and $G_{p,n}$ with $\gcd((s-1)p,n)=1$ for primes $p,q$) are special cases of $G_{m,n}^\star$. Hence, Theorems \ref{d(Gmnstar)} and \ref{inversed(G)} generalize the previous results obtained by Bass \cite{Bass2007},  Mart\'{\i}nez and Ribas \cite{BR2019}, Qu and Li \cite{QL, QL2022}.

\section{Notation and Preliminaries}
We follow the notation and conventions detailed in \cite{GG2013}. For real numbers $a,b\in \mathbb{R}$, we set $[a,b]=\{x\in \mathbb{Z}: a\leq x\leq b\}$. For integers $m,n\in \mathbb{Z}$, we denote by $\mbox{gcd}(m,n)$ the greatest common divisor of $m$ and $n$. Let $a$ and $n$ be relatively prime integers with $a\neq 0$ and $n$ positive. Then the least positive integer $k$ such that $a^k\equiv 1 {\pmod n}$ is called the {\sl order of $a$ modulo} $n$ and is denoted by ${\rm ord}_n a$.

Let $G$ be a finite group written multiplicatively with operation $*$. If there exists no ambiguity, we shall refer to $ab$ as the product of $a*b$ for any elements $a,b\in G$.
For a nonempty subset $X$ of $G$, we denote by  $\langle A \rangle $  the subgroup of $G$ generated by $X$. For any two subsets $A, B$ of $G$, we denote by $AB$ the set $\{a*b:a\in A,b\in B\}$, in particular, if $A$ or $B$ is a singleton, say $B=\{b\}$, we write $Ab$ to mean $A\{b\}$. More generally, for any integer $n\geq 2$  and any $n$ subsets $A_1, \ldots, A_n$ of $G$,  we denote by $A_1\ldots A_n$  the set $\{a_1*\cdots *a_n: a_i\in A_i \mbox{ for each } i\in [1,n]\}$. By a {\sl sequence} over a group $G$, we mean a finite, unordered sequence where repetition of elements is allowed. We view sequences over $G$ as elements of the free abelian monoid $\mathcal{F}(G)$, denote multiplication in $\mathcal{F}(G)$ by the bold symbol $\bm\cdot$ which differs from the
notation $*$ used for the multiplication of the group $G$, and we use brackets for all exponentiation in $\mathcal{F}(G)$. Then a sequence $S \in \mathcal F(G)$ can be written in the form $S= g_1  \bm \cdot \ldots \bm\cdot g_{\ell},$ where $|S|= \ell$ is the {\it length} of $S$. For any $g \in G$, we denote $\mathsf v_g(S) = |\{ i\in [1, \ell] : g_i =g \}|\,  $  the {\it multiplicity} of $g$ in $S$. Let $\mathsf h(S)=\max\limits_{g\in G}\{\mathsf v_g(S)\}$ be the maximum of the multiplicities of $S$. A sequence $T \in \mathcal F(G)$ is called a {\it subsequence} of $S$, denoted by $T \mid S$, if  $\mathsf v_g(T) \le \mathsf v_g(S)$ for every $g\in G$. Denote by $S \bm\cdot T^ {[-1]}$  the subsequence of $S$ obtained by deleting the terms of $T$ from $S$. If $S_1, S_2 \in \mathcal F(G)$, then $S_1 \bm\cdot S_2 \in \mathcal F(G)$ denotes the sequence satisfying $\mathsf v_g(S_1 \bm\cdot S_2) = \mathsf v_g(S_1 ) + \mathsf v_g( S_2)$ for every $g \in G$. For convenience, we  write
\begin{center}
 $g^{[k]} = \underbrace{g \bm\cdot \ldots \bm\cdot g}_{k} \in \mathcal F(G)\quad$
\end{center}
for any element $g \in G$ and any integer $k\geq 0$. For a sequence $S= g_1 \bm \cdot \ldots \bm\cdot g_{\ell} \in \mathcal F(G)$, we let $$\pi (S) = \{g_{\tau(1)}*\cdots *g_{\tau(\ell)}: \tau \mbox{ is a permutation of } [1, \ell] \} \subseteq G$$ be the {\it set of products} of $S$, and let $$\Pi_n(S) = \bigcup _{T\mid S,\ |T| = n}\pi(T)$$ be the {\it set of all $n$-products} of $S$, and let $$\Pi(S) = \bigcup_{1 \le n \le \ell}\Pi_n(S)$$ be the {\it set of all subsequence products} of $S$. The sequence $S$ is called
\begin{itemize}
\item[$\bullet$]  {\it product-one} if $1 \in \pi(S)$;
\item[$\bullet$] {\it product-one free} if $1\not\in \Pi(S)$;
\item[$\bullet$] {\it minimal product-one } if $1\in\pi(S)$ and $S$ can not be partitioned into two nontrivial, product-one subsequences.
\end{itemize}
Let $\mathbf{A}=(A_1,\ldots, A_{\ell})$ be a sequence of finite subsets of $G$, i.e., $A_i\subseteq G$ for each $i\in [1,\ell]$. Denote
$$\Pi(\mathbf{A})=\{a_{i_1}*\cdots *a_{i_{k}}:1\leq i_1<\cdots<i_{k}\leq \ell, k\in[1, \ell] \mbox{ and } a_{i_j}\in A_{i_j} \mbox{ for } j\in[1, k]\}.$$
For any subset $A$ of $G$,  we denote $\mbox{Stab}_G(A)=\{g\in G: gA=A\}$ to be its stabilizer in $G$ (the subgroup of $G$ fixing the set $A$). If there is no ambiguity, we shall just write ${\rm Stab}(A)$ for ${\rm Stab}_G(A)$.

Then we shall close this section with some useful lemmas.
The following one is the well known Kneser's Addition Theorem, and we direct the interested reader to \cite{GeK, Na} for a detailed proof.

\begin{lemma}\label{Kneser}\cite{GeK, Na} (Kneser's Addition Theorem) Let $A_1, \ldots, A_r$ be finite nonempty subsets of an abelian group $G$, and let $H=\mbox{\rm Stab}(A_1\ldots A_r)$. Then,
$$
|A_1\ldots A_r|\geq |A_1H|+\cdots +|A_rH|-(r-1)|H|.
$$
\end{lemma}

\begin{lemma}\cite[Lemma 2.8]{QL2022}\label{productonefree}
Let $\mathbf{A}=(A_1,A_2,\ldots,A_{\ell})$ be a sequence of subsets of a group $G$. If $1\notin \Pi(\mathbf{A})$, then $|\Pi(\mathbf{A})|\geq \sum_{i=1}^{\ell} |A_i|$.
\end{lemma}

\begin{lemma}\cite[Theorem 1.2]{GZ2006}\label{normalsequence}
Let $G$ be a finite abelian group, and let $S$ be a sequence over $G$. Let $T$ be a product-one subsequence of $S$ with maximal length. If $|T|=|S|-\mathsf d(G)$, then $\mathsf v_g(S\bm\cdot T^{[-1]})\geq 1$ for any $g\in G\setminus \{1\}$ such that $\mathsf v_g(S)\geq 1$.
\end{lemma}

\begin{lemma}\cite[Theorem 2.1(a)]{GGS2007}\label{n-1}
Let $G=\langle g\rangle$ be a finite cyclic group of order $n\geq 3$ generated by the element $g$, and let $S$ be a product-one free sequence over $G$ of length $|S|=n-1$. Then $S=(g^k)^{[n-1]}$ for some integer $k\in [1,n-1]$ such that $\gcd(k, n)=1$.
\end{lemma}

\section{Proof of main Theorems}
Throughout this section, we denote $G$ to be the group
\begin{equation}\label{equation G=Gm,n}
G_{m,n}^\star=\langle x, y| x^m=y^n=1, x^{-1}yx=y^s\rangle,
 \end{equation} with
\begin{equation}\label{equation ordqs=m for each q}
{\rm ord}_q s=m \mbox{ for every prime divisor } q \mbox{ of } n.
\end{equation}
  Let $K=\langle x\rangle$ and $N=\langle y\rangle$. Then $C_n\cong N\lhd G$ and $K\cong G/N\cong C_m$.
By \eqref{equation G=Gm,n}, we derive that
\begin{equation}\label{equation x-ayxa=ysa}
x^{-a}u x^a=u^{s^a} \mbox{ for any } u\in N  \mbox{ and any integer } a\geq 0.
\end{equation}
It follows that
\begin{equation}\label{equation s m power=1 mod n}
s^m\equiv 1 {\pmod n},
\end{equation} since $y=x^{-m} y x^m=y^{s^m}$. Let $\varphi$ be the canonical homomorphism of $G$ onto $G/N$. For any sequence $T=g_1\bm\cdot \ldots\bm\cdot g_{\ell}$ over $G$, we denote $\varphi(T)=\varphi(g_1)\bm\cdot \ldots\bm\cdot \varphi(g_{\ell})$. Then $\varphi(T)$ is a sequence over $G/N$. Note that {\sl if $\varphi(T)$ is a product-one sequence over $G/N$ then $\pi(T)\subseteq N$, this is because $G/N$ is abelian and so $\varphi(\pi(T))=\pi(\varphi(T))=\{1_{G/N}\}$.}

One fact worth noting is that for a cyclic group $C_n$ of order $n$, we have the following equality $\mathsf d(C_n)+1=\mathsf D(C_n)=n$, which shall be used later in this manuscript.

To prove the main theorems, we still need some lemmas.

\begin{lemma}\label{mequiv1}
Let $G=G_{m,n}^\star$ be given as above. Then the following conclusions hold.
\begin{itemize}
\item[(i)] $n\equiv1 \pmod m$.
\item[(ii)]  $\gcd(s^a-1,n)=1$ for every integer $a$ such that $m\nmid a$.
\item[(iii)] If $g^k\in N$ for some $g\in G\setminus N$, then $g^k=1$.
\end{itemize}
\end{lemma}

\begin{proof} (i)  By \eqref{equation ordqs=m for each q} and the Fermat's little theorem, we have that $m\mid q-1$ and so $q\equiv1 \pmod m$ for every prime divisor $q$ of $n$. This implies  $n\equiv1 \pmod m$, done.

(ii) Assume to the contrary that $\mbox{gcd}(s^{a}-1,n)\neq 1$ for some integer $a$ with $m\nmid a$. Then there exists a prime divisor $q$ of $n$ such that $s^{a}-1\equiv 0 \pmod q$. By \eqref{equation ordqs=m for each q}, we have $m\mid a$, yielding a contradiction.

(iii) Since $g\in G\setminus N$, without loss of generality we may assume that $g=x^{-a}h$ for some $a\in[1,m-1]$ and $h\in N$. Since $g^k\in N$ and $N\lhd G$, we derive that $m\mid ka$. Combined with \eqref{equation x-ayxa=ysa} and \eqref{equation s m power=1 mod n}, we have that
\begin{align*}
g^{k(s^a-1)}=& \left[(x^{-a}hx^a)*(x^{-2a}hx^{2a})*\cdots *(x^{-(k-1)a}hx^{(k-1)a})*(x^{-ka}hx^{ka})\right]^{s^a-1}\\
=& \left[(x^{-a}hx^a)*(x^{-2a}hx^{2a})*\cdots *(x^{-(k-1)a}hx^{(k-1)a})*h\right]^{s^a-1}\\
=& \left[h^{s^a}*h^{s^{2a}}*\cdots *h^{s^{(k-1)a}}*h\right]^{s^a-1}\\
=& \left[h^{s^a+s^{2a}+\cdots +s^{(k-1)a}+1}\right]^{s^a-1}\\
=&h^{s^{ka}-1}=1.
\end{align*}
By (ii), we see $\mbox{gcd}(s^a-1,n)=1$, i.e., there exist integers $\mu, \nu$ such that the integer $1=(s^a-1) \mu+n \nu$. Since $(g^k)^n=1$, it follows that $g^k=(g^k)^{(s^a-1) \mu+n \nu}=(g^{k(s^a-1)})^{\mu}(g^k)^{n \nu}=1$.
This completes the proof of the lemma.
\end{proof}



\begin{lemma}\label{basic} Let $M$ be a subgroup of $N=\langle y\rangle$, and let $u$ be an element of $N$.
\begin{itemize}
\item[(i)] If $u^{s^i}\in M$ for some $i\in [0, m-1]$, then $u\in M$.
\item[(ii)] If there exist two distinct integers $i,j\in [0,m-1]$ such that $u^{s^i}$ and $u^{s^j}$ belong to the same coset of $M$,  then $u\in M$.
\item[(iii)] If $u\neq 1$ then $u^{s^0},u^{s^1},\ldots,u^{s^{m-1}}$ are pairwise distinct nonidentity elements.
\end{itemize}
\end{lemma}
\begin{proof} (i) By \eqref{equation s m power=1 mod n}, we see $\mbox{gcd}(s,n)=1$, and so $\mbox{gcd}(s^i, n)=1$. Then there exist integers $\mu, \nu$ such that the integer $1=s^i \mu+n \nu$.  Since $u^n=u^{|N|}=1$, it follows that $u=u^{s^i \mu+n \nu}=(u^{s^i})^{\mu} u^{n\nu}=(u^{s^i})^{\mu}\in M$.

(ii) Since both $u^{s^i}$ and $u^{s^j}$ belong to the same coset of $M$, we have $u^{(s^{j-i}-1)s^i}=u^{s^j-s^i}\in M$. Since $m\nmid j-i$, it follows from Lemma~\ref{mequiv1}(ii) that $\mbox{gcd}(s^{j-i}-1,n)=1$, and thus $\mbox{gcd}((s^{j-i}-1)s^i,n)=1$. Similarly as in the argument for Conclusion (i), we can show that $u\in M$.

(iii) By taking  $M=\{1\}$ and applying Conclusion (ii), we have that $u^{s^0},u^{s^1},\ldots,u^{s^{m-1}}$ are pairwise distinct. Moreover, for any $i\in [1,m-1]$, since $\gcd({\rm ord}(u),s^{i})=1$, it follows that ${\rm ord}(u)\nmid s^{i}$ and so $u^{s^0},u^{s^1},\ldots,u^{s^{m-1}}$ are nonidentity elements.
\end{proof}

\begin{lemma}\label{conjugationone} Let $T=g_1\bm\cdot \ldots \bm\cdot g_t$ be a sequence over $G$ such that $\varphi (T)$ is a minimal product-one sequence over $G/N$.

(i) For any nonidentity element $u\in\pi(T)$, then there exist $r_1,\ldots, r_t\in [0,m-1]$ such that $u^{s^{r_{1}}}, u^{s^{r_{2}}},\ldots, u^{s^{r_{t}}}$ are pairwise distinct elements of $\pi(T)$;

(ii) For any subgroup $M$ of $N$, if $\pi(T)\cap M=\emptyset$ then $|\pi(T)M|\geq t|M|$.
\end{lemma}

\begin{proof} (i) Without loss of generality, we may assume that
\begin{equation}\label{equation u=g1timestogt}
u=g_1*\cdots *g_t.
\end{equation}
 Since $\varphi(T)$ is a minimal product-one sequence over $G/N\cong C_m$ of length $t$, we conclude that $t\leq {\rm D}(G/N)=m$ and
 \begin{equation}\label{equation u in N}
 u\in N.
 \end{equation}
  For each $i\in [1,t]$, we write
\begin{equation}\label{equation product of gi}
g_1*\cdots *g_i=x^{r_i}y^{c_i}
\end{equation}
 with $r_i\in[0,m-1]$ and $c_i\in [0, n-1]$.  Let $u_i=(g_{1}*\cdots *g_{i})^{-1}u(g_1*\cdots *g_{i})$ for each $i\in [1,t]$. It follows from \eqref{equation u=g1timestogt} that $u_i\in \pi(T)$, and follows from \eqref{equation x-ayxa=ysa}, \eqref{equation u in N} and \eqref{equation product of gi} that
$u_i=(x^{r_i}y^{c_i})^{-1} u (x^{r_i}y^{c_i})=y^{-c_i}(x^{-r_i} u x^{r_i})y^{c_i}=y^{-c_i} u^{s^{r_i}} y^{c_i}=u^{s^{r_i}}$,
where $i\in [1,t]$.

Then it suffices to show that $u^{s^{r_{1}}}, u^{s^{r_{2}}},\ldots, u^{s^{r_{t}}}$ are pairwise distinct. Assume to the contrary that $u^{s^{r_{i}}}=u^{s^{r_{j}}}$ for some $1\leq i<j\leq t$. By Lemma \ref{basic} (iii), we see that $r_i=r_j$. Then $g_{i+1}*\cdots *g_j=(g_1*\cdots * g_i)^{-1}(g_1*\cdots * g_j)=(x^{r_i}y^{c_i})^{-1}x^{r_j}y^{c_j}=y^{c_j-c_i}\in N$. Thus $\varphi(g_{i+1})*\cdots * \varphi(g_j)=\varphi(g_{i+1}*\cdots * g_j)=1$, yielding a contradiction to the assumption that $\varphi(T)$ is a minimal product-one sequence over $G/N$.

(ii) Since $\varphi(T)$ is a minimal product-one sequence over $G/N$, it follows that $\pi(T)\subseteq N$. Take $u\in \pi(T)$. Since $u\notin M$, it follows from Conclusion (i) and Lemma \ref{basic} (ii) that $u^{s^{r_{1}}}, u^{s^{r_{2}}},\ldots, u^{s^{r_{t}}}\in \pi(T)$ belongs to the distinct cosets of $M$, where $r_1,\ldots, r_t\in [0,m-1]$. This implies that $|\pi(T)M|\geq t|M|$, done.
\end{proof}

\begin{lemma}\label{conjugationtwo}
Let $T$ be a sequence of $m$ terms from the coset $x^aN$ of $N$
for some integer $a\in [1,m-1]$ with $\gcd(a,m)=1$. For any $u\in N$, we have  $\bigcup\limits_{i=0}^{m-1}\left(\pi(T)u^{s^{i}}\right)\subseteq \pi(T\bm\cdot u)$.
\end{lemma}
\begin{proof} Let $T=g_1\bm\cdot \ldots \bm\cdot g_m$, where $g_i=x^ay^{c_i}$ and $c_1,\ldots,c_{m}\in [0, n-1]$. Take an arbitrary element $v\in \pi(T)$, say $v=g_1* \cdots *g_m$.
Let $u_i=g_1* \cdots *g_iug_{i+1}* \cdots *g_m$ for each $i\in [1, m]$ (in particular, $u_m=g_1* \cdots *g_mu$).  Observe that $$u_i\in \pi(T\bm\cdot u) \mbox{ for each } i\in [1,m].$$  Combined with \eqref{equation x-ayxa=ysa}, we have $u_i=g_1 * \cdots * g_iug_{i+1} * \cdots * g_m=v[(g_{i+1}* \cdots * g_m)^{-1}u(g_{i+1} * \cdots * g_m)]=vu^{s^{(m-i)a}}$ for $i\in [1,m]$, and so
all the elements
\begin{equation}\label{equation vus0to m-1 belong to}
vu^{s^{0a}}, vu^{s^{1a}},\ldots,vu^{s^{(m-1)a}}\in \pi(T\bm\cdot u).
\end{equation}
On the other hand, since $\gcd(a,m)=1$, then $0a,1a,\ldots,(m-1)a$ is a complete system of residues modulo $m$, it follows from \eqref{equation s m power=1 mod n} that $s^{0a},s^{1a},\ldots,s^{(m-1)a}$, taken in some order, must be congruent to
 $s^{0},s^{1},\ldots,s^{m-1}$ modulo $n$ one to one. Since ${\rm ord}(u)\mid n$, it follows from \eqref{equation vus0to m-1 belong to} that $vu^{s^{i}}\in \pi(T\bm\cdot u)$ for each  $i\in [0,m-1]$. By the arbitrary choice of $v\in \pi(T)$, we have the lemma proved.
\end{proof}

\begin{lemma}\label{PIS}  Let $S$ be a product-one free sequence over $G$, and let $\varphi$ be the canonical epimorphism of $G$ onto $G/N$. If $\varphi (S)$ is a product-one sequence over $G/N$,  then $|\Pi(S)\cap N|\geq |S|$.
\end{lemma}
\begin{proof} Since $\varphi (S)$ is a product-one sequence over $G/N$, we can obtain a factorization $S=S_1\bm\cdot \ldots \bm\cdot S_t$ such that $\varphi(S_i)$ is a minimal product-one subsequence over $G/N$ for each $i\in [1,t]$, where $t\geq 1$. For each $i\in [1,t]$, we take $u_i\in \pi(S_i)$, since $S$ is product-one free over $G$, we have $1\notin\pi(S_i)$ and so
$u_i\neq 1$. Combined with Lemma \ref{conjugationone} (i), we have that
\begin{equation}\label{equation cad pi(Si) geq cad Si}
|\pi(S_i)|\geq  |S_i| \mbox{ for each } i\in [1,t].
\end{equation}
Let $\mathbf{A}=\left(\pi(S_1),\pi(S_2),\ldots,\pi(S_t)\right)$. Since $\varphi(S_i)$ is a minimal product-one subsequence over $G/N$, it follows that $\pi(S_i)\subseteq N$ for each $i\in [1,t]$. This implies that $\Pi(\mathbf{A})\subseteq \Pi(S)\cap N$. Since $1\notin \Pi(S)$, then $1\notin \Pi(\mathbf{A})$. By \eqref{equation cad pi(Si) geq cad Si}, Lemma~\ref{productonefree}, we have that $|\Pi(S)\cap N|\geq |\Pi(\mathbf{A})|\geq \sum_{i=1}^{t} |\pi(S_i)|\geq \sum_{i=1}^{t}|S_i|=|S|$. This completes the proof of the lemma.
\end{proof}

Now we are in a position to prove our main theorem.

\vspace{0.5cm}

\noindent {\bf Proof of Theorem \ref{d(Gmnstar)}} \ Let $S_0=x^{[m-1]}\bm\cdot y^{[n-1]}$. It is easy to check that $S_0$ is a product-one free sequence over $G$. Then $\mathsf d(G)\geq |S_0|=m+n-2$ follows readily.
Now we take a sequence $S$ over $G$ of length $m+n-1$. It suffices to show that $1\in\Pi(S)$. Suppose to the contrary that
\begin{equation}\label{equation 1 notin Pi(S)}
1\notin\Pi(S).
\end{equation}
Since $|\varphi(S)|=|S|\geq m=\mathsf d(C_m)+1=\mathsf d(G/N)+1$, we can take a subsequence $T\mid S$ such that $\varphi(T)$ is a {\sl product-one sequence over $G/N$} with maximal length.
Then $\varphi(S\bm\cdot T^{[-1]})$ is a product-one free sequence over $G/N$, and thus $|S\bm\cdot T^{[-1]}|=|\varphi(S\bm\cdot T^{[-1]})|\leq \mathsf d(G/N)=m-1$.
It follows from \eqref{equation 1 notin Pi(S)} that $T$ is product-one free. By applying Lemma~\ref{PIS} with the sequence $T$,  we have that $|\Pi(T)\cap N|\geq |T|=|S|-|S\bm\cdot T^{[-1]}|\geq (m+n-1)-(m-1)=n=|N|$, which implies that $N=N\cap \Pi(T)$ and so $1\in \Pi(T)$, a contradiction with $T$ being product-one free. Therefore, $1\in\Pi(S)$ and the proof is completed.                                   \qed

\noindent {\bf Proof of Theorem \ref{inversed(G)}} \ We first show the necessity, i.e., to prove that if $S$ is of the given form as in (a) or (b), then it is product-one free. If $S$ is of the form given as in (b), the conclusion is easily to verify. Hence, we need only to consider the case that
$S$ is of the form given as in (a). Let $S'$ be an arbitrary nonempty subsequence of $S$. It suffices to show that $1\notin \pi(S')$.

Let $S'=S_1'\cdot S_2'$ where $S_1'\mid (x^ay^{b_1})\bm\cdot\ldots\bm\cdot (x^ay^{b_{m-1}})$ and $S_2'\mid (y^c)^{[n-1]}$. Suppose $S_1'\neq \emptyset$. Since $\gcd(a,m)=1$ and $|S_1'|\leq m-1$, it follows that $\varphi(S')$ is not a product-one sequence over $G/N$, i.e. $1_{G/N}\notin\varphi(\pi(S'))=\pi(\varphi(S'))$ and so $\pi(S')\cap N=\emptyset$. Then, $1\notin\pi(S')$ follows readily. Otherwise $S_1'=\emptyset$. Then $S'\mid (y^c)^{[n-1]}$. Clearly, $1\notin\pi(S')$ since $\mbox{gcd}(c,n)=1$. This proves the necessity of the theorem.

Now we show the sufficiency, i.e., to prove that if $S$ is a product-one free sequence over $G$ of length
\begin{equation}\label{equation length|S|=m+n-2}
|S|=m+n-2,
\end{equation}
 then it is of the given form as (a) or (b). Let $T$ be a subsequence of $S$  such that $\varphi(T)$ is a product-one sequence over $G/N$ with maximal length. Since $\varphi(S\bm\cdot T^{[-1]})$ is a  product-one free sequence over $G/N$, it follows that $|S\bm\cdot T^{[-1]}|=|\varphi(S\bm\cdot T^{[-1]})|\leq \mathsf d(G/N)=m-1$. Since $T$ is product-one free over $G$,
it follows from Lemma \ref{PIS} that $n-1=|N\setminus \{1\}|\geq |\Pi(T)\cap N|\geq  |T|=|S|-|S\bm\cdot T^{[-1]}|\geq (m+n-1)-(m-1)= n-1$, which implies
\begin{equation}\label{equation ST-1=m-1}
|S\bm\cdot T^{[-1]}|=m-1=\mathsf d(G/N).
\end{equation}
By Lemma~\ref{n-1}, we have that $\varphi(S\bm\cdot T^{[-1]})=(x^aN)^{[m-1]}$ for some $a\in [1,m-1]$ with
\begin{equation}\label{equation gcdam=1}
\gcd(a,m)=1.
\end{equation}
 Moreover, by Lemma~\ref{normalsequence}, we have that
\begin{equation}\label{equation all terms from N or}
\supp(\varphi(S))\subseteq \{N, x^aN\},
 \end{equation} i.e., all terms of $S$ are from the subgroup $N$ or from the coset $x^aN$. Let $S_N$ be the subsequence of $S$ consisting of all terms of $S$ from $N$, and let $W=S\cdot S_N^{[-1]}.$ It follows from \eqref{equation all terms from N or} that $\supp(\varphi(W))=\{x^aN\}$. Combined with  \eqref{equation ST-1=m-1}, \eqref{equation gcdam=1} and the maximality of $|T|$, we conclude that
 \begin{equation}\label{equation form of W}
 W=(x^ah_1)\bm\cdot\ldots \bm\cdot (x^ah_{t}),
 \end{equation}
 where  $h_1,\ldots,h_t\in N$, and
  \begin{equation}\label{equation t=length of W}
 t=|W|=|\varphi(W)|=km+m-1
 \end{equation}
  for some $k\in [0, (n-1)/m]$.

Now we assume $k=0$. Then $|S_N|=|S|-|W|=n-1=|N|-1$. Since $S_N$ is a product-one free sequence over $N$, it follows from Lemma \ref{n-1} that
$S_N=(y^c)^{[n-1]}$  with $\gcd(c,n)=1$. Then we obtain $S=W\bm\cdot S_N=(x^ay^{b_1})\bm\cdot\ldots\bm\cdot (x^ay^{b_{m-1}})\bm\cdot (y^c)^{[n-1]}$
as desired. Hence, we need only to consider the case of $$k\geq 1,$$
and we shall prove that $G\cong G_{2,3}$ and $S=x\bm\cdot xy\bm\cdot xy^2$.
\vspace{0.5cm}

\noindent {\bf Claim A.} $S=W$, $|S|=(k+1)m-1$ and $n=km+1$.
\vspace{0.5cm}

\noindent{\sl Proof of Claim A.} \ To prove Claim A, by \eqref{equation length|S|=m+n-2} and \eqref{equation t=length of W}, it is sufficient to show that $S=W$. Suppose to the contrary that $S\neq W$, i.e., $|S_N|\geq 1$.  Take $u\mid S_N$. Then $u\neq 1$. Let $T_1$ be a subsequence of $W$ with
\begin{equation}\label{equation length T1=m}
|T_1|=m.
\end{equation} Since $\gcd(a,m)=1$, it follows from \eqref{equation form of W}, \eqref{equation length T1=m} that $\varphi(T_1)$ is a minimal product-one sequence over $G/N$ and $\pi(T_1)=\Pi(T_1)\cap N$.
Since $T_1$ is a product-one free sequence over $G$, it follows from Lemma  \ref{PIS} that
\begin{equation}\label{equation card pi(T1) geq T1}
|\pi(T_1)|=|\Pi(T_1)\cap N|\geq |T_1|=m.
\end{equation}
Since $u\in N$, it follows from \eqref{equation length T1=m} that
$\pi(T_1)\cup \pi(T_1\bm\cdot u)\subseteq N.$
By Lemma~\ref{basic}~(iii), we have that $u^{s^0}, u^{s^{1}},\ldots, u^{s^{m-1}}$ are pairwise distinct, and combined with Lemma~\ref{conjugationtwo}, we have that
\begin{equation}\label{equation piprowith1uetc}
\Pi(T_1\bm\cdot u)\cap N\supseteq \pi(T_1)\cup \left(\bigcup\limits_{i=0}^{m-1}\left(\pi(T_1)u^{s^{i}}\right)\right)=\pi(T_1)\{1,u, u^{s},\ldots, u^{s^{m-1}}\}.
\end{equation}

Now we show that \begin{equation}\label{equation cadpi(T1u)geq 2m}
|\Pi(T_1\bm\cdot u)\cap N|\geq 2m.
\end{equation}
Let $M={\rm Stab}\left(\pi(T_1)\{1,u, u^{s},\ldots, u^{s^{m-1}}\}\right)$. It follows from \eqref{equation piprowith1uetc} that $M<N$.  Suppose that $M=\{1\}$. By \eqref{equation card pi(T1) geq T1}, \eqref{equation piprowith1uetc} and Lemma \ref{Kneser}, we have that
 $|\Pi(T_1\bm\cdot u)\cap N|\geq |\pi(T_1)\{1,u, u^{s},\ldots, u^{s^{m-1}}\}|\geq |\pi(T_1)|+|\{1,u, u^{s},\ldots, u^{s^{m-1}}\}|-1\geq 2m$, done.  Otherwise $|M|>1$. Since $1\notin \Pi(T_1\bm\cdot u)$, it follows from \eqref{equation piprowith1uetc} that $1\notin\pi(T_1)\{1, u, u^{s},\ldots, u^{s^{m-1}}\}$, which implies that $\pi(T_1)\cap M=\emptyset$.
Recall that $\varphi(T_1)$ is a minimal product-one sequence over $G/N$. Combined with \eqref{equation length T1=m}, \eqref{equation piprowith1uetc}, Lemma \ref{Kneser} and Lemma \ref{conjugationone} (ii), we have that
 $|\Pi(T_1\bm\cdot u)\cap N|\geq |\pi(T_1)\{1,u, u^{s},\ldots, u^{s^{m-1}}\}|\geq |\pi(T_1)M|+|\{1,u, u^{s},\ldots, u^{s^{m-1}}\}M|-|M|\geq m|M|+|M|-|M|\geq 2m$. Therefore, we have \eqref{equation cadpi(T1u)geq 2m} proved.

Let $T_2$  be a subsequence of $W\bm\cdot T_1^{[-1]}$ with
\begin{equation}\label{equation length T2=(k-1)m}
|T_2|=(k-1)m.
\end{equation}
 Notice that $(S_N\bm\cdot u^{[-1]})\bm\cdot T_2$ is a product-one free sequence over $G$ and $\varphi\left((S_N\bm\cdot u^{[-1]})\bm\cdot T_2\right)$ is
 a product-one sequence over $G/N$. By Lemma~\ref{PIS}, we have
 \begin{equation}\label{equation pi SnuT2capNgeq}
 |\Pi\left((S_N\bm\cdot u^{[-1]})\bm\cdot T_2\right) \cap N|\geq |(S_N\bm\cdot u^{[-1]})\bm\cdot T_2|.
 \end{equation}
Let ${\bf B}=(B_1,B_2)$ be a sequence of subsets of $G$, where $B_1=\Pi(T_1\bm\cdot u)\cap N$ and $B_2=\Pi\left((S_N\bm\cdot u^{[-1]})\bm\cdot T_2\right) \cap N$. Observe that
$\Pi(S)\cap N\supseteq \Pi({\bf B}).$
Since $1\notin \Pi(S)$, it follows from \eqref{equation t=length of W}, \eqref{equation cadpi(T1u)geq 2m}, \eqref{equation length T2=(k-1)m}, \eqref{equation pi SnuT2capNgeq} and Lemma~\ref{productonefree} that
\begin{align*}
|\Pi(S)\cap N|\geq |\Pi({\bf B})|&\geq |\Pi(T_1\bm\cdot u)\cap N|+|\Pi((S_N\bm\cdot u^{[-1]})\bm\cdot T_2) \cap N|\\
&\geq |\Pi(T_1\bm\cdot u)\cap N|+|(S_N\bm\cdot u^{[-1]})\bm\cdot T_2|\\
&\geq 2m+(|S_N|-1)+(k-1)m\\
&=(km+m-1)+|S_N|\\
&=|W|+|S_N|\\
&=|T|\\
&=m+n-2\geq n=|N|,
\end{align*}
which implies that $N\subseteq \Pi(S)$, a contradiction. This completes the proof of the claim. \qed

By Claim A, we can fix a factorization of $W$ given as
\begin{equation}\label{equation factorization of W}
W=T_0\bm\cdot T_1\bm\cdot \ldots \bm\cdot T_{k-1},
\end{equation} where $|T_0|=2m-1$ and $|T_i|=m$ for $i\in [1,k-1]$.
 By taking some subsequence $T'\mid T_0$ of length  $m$ and by applying Lemma \ref{conjugationone} (ii) with $M=\{1\}$, we conclude that
  \begin{equation}\label{equation 1 geq m}
 |\Pi_m(T_0)|\geq |\pi(T')|=|\pi(T')M|\geq  m.
  \end{equation}
  By  applying Lemma~\ref{PIS}, we have
    \begin{equation}\label{equation 2 geq (k-1)m}
  |\Pi(W\bm\cdot T_0^{[-1]})\cap N|\geq |W\bm\cdot T_0^{[-1]}|=(k-1)m.
    \end{equation}
 Let ${\bf C}=(C_1,C_2)$ be a sequence of subsets of $G$, where $C_1=\Pi_m(T_0)$ and $C_2=\Pi(W\bm\cdot T_0^{[-1]})\cap N$. Note that
$\Pi({\bf C})\subseteq \Pi(W)\cap N.$ Then by \eqref{equation 1 geq m}, \eqref{equation 2 geq (k-1)m}, Claim A and Lemma~\ref{productonefree}, we derive that $km=n-1=|N\setminus \{1\}| \geq |\Pi(W)\cap N|\geq |\Pi({\bf C})|\geq |\Pi_m(T_0)|+|\Pi(W\bm\cdot T_0^{[-1]})\cap N|\geq m+(k-1)m=km$, which implies that
\begin{equation}\label{equation car Pim(T0)=m}
|\Pi_m(T_0)|= m.
\end{equation}

Since $S$ is product-one free, it follows from Lemma~\ref{mequiv1}~(iii) that $\mathsf h(W)\leq m-1,$
and so $\mathsf h(T_0) \leq m-1$. Then we can obtain a factorization of the sequence $T_0$ with
\begin{equation}\label{equation T0=A1dotsAm-1}
T_0=A_1\bm\cdot \ldots \bm\cdot A_{m-1}
\end{equation}
such that $|A_1|=3$ and $|A_2|=\cdots=|A_{m-1}|=2$, and in particular each subsequence $A_i$ is {\sl square-free} (a set) for $i\in [1,m-1]$. Observe that $$\Pi_m(T_0)\supseteq\Pi_2(A_1) A_2 \ldots A_{m-1}.$$ Without loss of generality, we may assume that $A_1=\{g_1,g_2,g_3\}$ and $A_i=\{g_{2i}, g_{2i+1}\}$ for each $i\in [2,m-1]$. Recall from
\eqref{equation form of W} that $$g_j=x^a h_j \mbox{ where } h_j\in N$$ for each $j\in [1,2m-1]$.
By \eqref{equation x-ayxa=ysa}, we can easily check that the set $\Pi_2(A_1)$ consists of the following elements: $$x^{2a}h_1^{s^{a}}h_2, x^{2a}h_1^{s^{a}}h_3, x^{2a}h_2^{s^{a}}h_1, x^{2a}h_2^{s^{a}}h_3, x^{2a}h_3^{s^{a}}h_1, x^{2a}h_3^{s^{a}}h_2$$ (not necessarily pairwise distinct). Now, we let $B_1,\ldots, B_{m-1}$ be subsets of $N$, where
$B_1=\supp\big{(}(h_1^{s^{(m-1)a}}h_2^{s^{(m-2)a}})\bm\cdot (h_1^{s^{(m-1)a}}h_3^{s^{(m-2)a}})$
$\bm\cdot (h_2^{s^{(m-1)a}}h_1^{s^{(m-2)a}})\bm\cdot
(h_2^{s^{(m-1)a}}h_3^{s^{(m-2)a}})\bm\cdot (h_3^{s^{(m-1)a}}h_1^{s^{(m-2)a}})\bm\cdot (h_3^{s^{(m-1)a}}h_2^{s^{(m-2)a}})\big{)}$
and $B_i=\supp\left(h_{2i}^{s^{(m-(i+1))a}}\bm \cdot h_{2i+1}^{s^{(m-(i+1))a}}\right)$  for each $i\in [2,m-1].$
A straightforward verification shows that
$$B_1\ldots B_{m-1}=\Pi_2(A_1)A_2\ldots A_{m-1}\subseteq \Pi_m(T_0).$$
Combined with \eqref{equation car Pim(T0)=m}, we have
\begin{equation}\label{equation card product leq m}
|B_1\ldots B_{m-1}|\leq m.
\end{equation}
Since $h_{2i}=x^{-a} g_{2i}\neq x^{-a} g_{2i+1}= h_{2i+1}$ and $\gcd\left(s^{(m-(i+1))a},n\right)=1$, it follows that $h_{2i}^{s^{(m-(i+1))a}}\neq h_{2i+1}^{s^{(m-(i+1))a}}$ for $i\in [1,m-1]$. This implies that \begin{equation}\label{equation |B1|geq 2}
 |B_1|\geq 2
 \end{equation}
and \begin{equation}\label{equation |Bi|geq 2}
  |B_i|=2, \mbox { for each } i\in [2,m-1].
   \end{equation}
Denote $M=\mbox{Stab}(B_1\ldots B_{m-1})$. Suppose $M\neq \{1\}$. Then $|M|> m$, this is because that $m|q-1$ and so $m<q$ for every prime divisor $q$ of $n$ (see the argument of Lemma \ref{mequiv1}).  By Lemma~\ref{Kneser}, we have that $|B_1\ldots B_{m-1}|\geq |M|> m$, which is a contradiction with \eqref{equation card product leq m}. Hence, $M=\{1\}.$ By \eqref{equation |B1|geq 2}, \eqref{equation |Bi|geq 2} and Lemma~\ref{Kneser}, we have that $|B_1\ldots B_{m-1}|\geq \sum_{i=1}^{m-1}|B_i|-(m-2)\geq 2(m-1)-(m-2)=m$. Combined with \eqref{equation card product leq m}, we have that $|B_1\ldots B_{m-1}|=m$, which implies that $|B_1|=2$, and in particular,
\begin{equation}\label{equation three equations}
h_1^{s^{(m-1)a}}h_2^{s^{(m-2)a}}=h_2^{s^{(m-1)a}}h_3^{s^{(m-2)a}}=h_3^{s^{(m-1)a}}h_1^{s^{(m-2)a}}
\end{equation}
and
\begin{equation}\label{equation three equations two}
h_1^{s^{(m-1)a}}h_3^{s^{(m-2)a}}=h_2^{s^{(m-1)a}}h_1^{s^{(m-2)a}}=h_3^{s^{(m-1)a}}h_2^{s^{(m-2)a}}.
\end{equation}
Since $h_1^{s^{(m-1)a}}h_2^{s^{(m-2)a}}=h_2^{s^{(m-1)a}}h_3^{s^{(m-2)a}}$, it follows that
$(h_1h_2^{-1})^{s^{(m-1)a}}=h_1^{s^{(m-1)a}}h_2^{-s^{(m-1)a}}=h_2^{-s^{(m-2)a}}h_3^{s^{(m-2)a}}=(h_3h_2^{-1})^{s^{(m-2)a}}$, and follows from \eqref{equation s m power=1 mod n} that
\begin{equation}\label{equation instant 1}
(h_1h_2^{-1})^{s^a}=[(h_1h_2^{-1})^{s^{(m-1)a}}]^{s^{2a}}=[(h_3h_2^{-1})^{s^{(m-2)a}}]^{s^{2a}}=h_3h_2^{-1}.
\end{equation}
Since $h_3^{s^{(m-1)a}}h_2^{s^{(m-2)a}}=h_2^{s^{(m-1)a}}h_1^{s^{(m-2)a}}$, it follows from the similar argument as above that
\begin{equation}\label{equation instant 2}
(h_3h_2^{-1})^{s^a}=h_1h_2^{-1}.
\end{equation}
By \eqref{equation instant 1} and \eqref{equation instant 2}, we conclude that
$$(h_1h_2^{-1})^{s^{2a}-1}=1.$$
If $m\nmid 2a$, then by Lemma~\ref{mequiv1}~(ii), we have that $\gcd(s^{2a}-1,n)=1$, which implies $h_1=h_2$, a contradiction. Hence, $m\mid 2a$. Since $\gcd(a,m)=1$, it follows that $$m=2$$ and $a=1$. By \eqref{equation T0=A1dotsAm-1}, we have that
\begin{equation}\label{equation last T0=A1}
T_0=A_1.
\end{equation}
By \eqref{equation s m power=1 mod n} and Lemma \ref{mequiv1} (ii), we see that $(s+1)(s-1)=s^2-1=s^m-1\equiv 0\pmod n$ and $\gcd(s-1,n)=1$, which implies $$s^a=s\equiv -1 \pmod {n}.$$  Combined with \eqref{equation three equations} and \eqref{equation three equations two}, we have that $h_1^{-1}h_2=h_2^{-1}h_3=h_3^{-1}h_1$ and $h_1^{-1}h_3=h_2^{-1}h_1=h_3^{-1}h_2$, which implies that $(h_1^{-1}h_2)^2=(h_2^{-1}h_3)(h_3^{-1}h_1)=(h_1^{-1}h_2)^{-1}$ and so $(h_1^{-1}h_2)^3=1$, and furthermore, we have that
\begin{equation}\label{equuation pi2(T0)=Hsetminus1}
\Pi_2(T_0)=\Pi_2(A_1)=B_1=\{h_1^{-1}h_2, h_2^{-1}h_1\}=\langle h_1^{-1}h_2\rangle\setminus{\{1\}}
\end{equation}
 and
\begin{equation}\label{equuation 3 subgroup}
|\langle h_1^{-1}h_2\rangle|=3.
\end{equation} Since $\langle h_1^{-1}h_2\rangle$ is a subgroup of $N$, it follows that $$3\mid n.$$

Suppose that $n\neq 3$. Then $n\geq 6$, and by Claim A, we have that $k\geq 3$ and $|W|\geq 7$. Similarly as \eqref{equation factorization of W}, we can choose a factorization of $W$ given as $W=L_0\bm\cdot L_1\bm\cdot \ldots \bm\cdot L_{k-1},$ with $|L_0|=2m-1$ and $|L_i|=m$ for $i\in [1,k-1]$, and moreover, such that $L_0$ and $T_0$ are {\sl disjoint} subsequences of $W$. By the arbitrariness of factorization given in \eqref{equation factorization of W}, we have that
\begin{equation}\label{equuation pi2(L0)=Hsetminus2}
\Pi_2(L_0)=H\setminus \{1\}
\end{equation}
 where $H$ is a subgroup of $N$ with order $|H|=3$. Since $N$ is cyclic, it follows from \eqref{equuation 3 subgroup} that $\langle h_1^{-1}h_2\rangle=H$ is the unique subgroup of $N$ of order three. Since $L_0$ and $T_0$ are  disjoint subsequences of $W$, it follows from \eqref{equuation pi2(T0)=Hsetminus1} and \eqref{equuation pi2(L0)=Hsetminus2} that $H=\Pi_2(T_0)\Pi_2(L_0)\subseteq \Pi(W)$, which is a contradiction with $S (=W)$  being product-one free. Hence,  $$n=3.$$ Then, $G\cong C_2\ltimes C_3$. It follows from  Claim A and \eqref{equation factorization of W} that $k=1$ and $S=W=T_0$. Combined with \eqref{equation form of W} and \eqref{equation last T0=A1}, we conclude that $S$ is square-free and thus, $S=x\bm\cdot xy\bm\cdot xy^2$.    This completes the proof of the theorem.                   \qed

\subsection*{Acknowledgements}
This work is supported by NSFC (grant no. 12371335),  the National Natural Science Foundation of Henan (grant no. 242300421393), the Foundation of Henan Educational Committee (grant no. 23A110012).


\end{document}